\documentclass[12pt]{article}
\usepackage{amssymb}
\usepackage{graphicx}
\usepackage{latexsym}
\def\part#1{\frac{\partial\phantom{q}}{\partial#1}}

\newenvironment{proof}{\begin{trivlist}\item[]{\bf Proof:} }
{\hfill $\Box$ \end{trivlist}}

\newtheorem{thm}{Theorem}

\newcommand{\lie}[1]{\mathfrak{#1}}

\def\Hom{\mathop{\rm Hom}\nolimits}

\def\tr{\mathop{\rm tr}\nolimits}

\def\ad{\mathop{\rm ad}\nolimits}

\newcommand{\R}{\mathbf{R}}
\newcommand{\C}{\mathbf{C}}

\newcommand{\PP}{{\rm P}}
\newcommand{\SB}{{\mathbf S}}

\begin{document}
\title{A remark on calibrations and Lie groups}
\author{Nigel Hitchin\\[5pt]}
\maketitle

\centerline{{\it Dedicated to Blaine Lawson on the occasion of his 80th birthday}}

\begin{abstract}
We use the notion of the principal three-dimensional subgroup of a simple Lie group to identify certain special subspaces of the Lie algebra and address the question of whether these are calibrated for invariant forms on the group.
\end{abstract}

\section{Introduction}

The notion of a calibrated differential form $\varphi$, as introduced in \cite{HL}, has become very important especially in the study of Calabi-Yau, $G_2$ and $Spin(7)$-manifolds, where $\varphi$  is a covariant constant form. On the other hand, the manifolds which have most covariant constant forms, namely compact simple Lie groups $G$, have received less attention, although they are addressed in \cite{T},\cite{XL},\cite{XL1},\cite{R}. 

Recall that the cohomology of a simple Lie group $G$ of rank $\ell$ is an exterior algebra on  $\ell$  generators with harmonic representatives $\varphi_i$ of odd degree $d_i$ which are covariant constant. The Cartan 3-form $\varphi_1$ is the generator of smallest degree and Tasaki \cite{T} showed that this defines a calibration and moreover that a three-dimensional subgroup associated to the highest root is  calibrated for this form and is volume-minimizing. He also showed that the Hodge dual  $\ast \varphi_1$ calibrates the codimension $3$ subspace of non-regular elements of $G$. 

Amongst the three-dimensional subgroups there is a particularly distinguished one, the {\it principal} three-dimensional subgroup, and Kostant showed \cite{K} that under the action of this group the Lie algebra  decomposes $\lie{g}=V_1\oplus V_2\oplus \cdots \oplus V_{\ell}$ into irreducible representations of $SO(3)$ whose dimensions are precisely the degrees $d_i$ of the generators of the cohomology. The author conjectured in \cite{H} that there is an exact fit here -- that for each subspace $V_i$ there exists a corresponding generator which restricts nontrivially. To the author's knowledge this has not yet been confirmed, though there is some information in \cite{BK}. In any case, if the restriction is non-zero it opens up the possibility of more complex calibrated submanifolds.

In this paper we observe first that the function defined by $\varphi_i$ on the Grassmannian of oriented subspaces of $\lie{g}$ of dimension $d_i$ has a critical point on $V_i$. If this critical value is nonzero then any submanifold of dimension $d_i$ tangential to a conjugate of $V_i$ will be minimal \cite{R}. If the non-zero value is the maximum then $\varphi_i$ defines a calibration and any such  submanifold is volume minimizing. 

 We  then search for non-zero values by using the transitive action of groups on odd-dimensional spheres $S^{2m+1}$, and an argument initiated by X.Liu  \cite{XL}. This consists of pulling back the volume form on the sphere and averaging over the group to produce an invariant form on $G$ of  degree $2m+1$. We use the well-known list of  groups with transitive actions  to show that in each case the pull-back of the volume form restricted to a corresponding $V_i$ is non-negative and hence its average is non-zero, providing some evidence for the conjecture. The relevant degrees are 
$2n-1$ for $SO(2n)$ and $SU(n)$, $4n-1$ for $Sp(n)$, $7$ for $Spin(7)$ and $15$ for $Spin(9)$.

Finally we mention the entirely different context \cite{H} in which the conjecture arose, involving the moduli space of stable bundles on a curve $C$.
\section{Invariant forms}
Let $G$ be a compact simple Lie group. The covariant constant forms on $G$ are the bi-invariant forms and these are defined as multilinear alternating forms $\alpha$ on $\lie{g}$ by
$$\alpha(a_1,\dots, a_{2m+1})=p(a_1,[a_2,a_3],......[a_{2m},a_{2m+1}])$$
where $p$ is an adjoint-invariant  polynomial of degree $m+1$. These polynomials correspond under the Chern-Weil homomorphism to characteristic classes like Chern or Pontryagin classes and we shall often  label the invariant forms this way -- as classes of degree $2m+2$ in the cohomology $H^*(B_G)$ of the classifying space. The Killing form is a quadratic polynomial and yields the Cartan 3-form.

The irreducible representations of the three-dimensional group $SU(2)$ are symmetric powers $\SB^n$ of the standard complex 2-dimensional representation $\SB$. The space $\SB^n$ may be thought of as the action on homogeneous polynomials $p(z_1,z_2)$ of degree $n$, or more conveniently the polynomial  $p(z)=p(z_1/z_2,1)$ and is therefore of dimension $n+1$. Since $-1\in SU(2)$ acts trivially if $n$ is even, these are the irreducibles for $SO(3)$ and are real. When $n$ is odd they are quaternionic representations of $SU(2)$. 

The Clebsch-Gordon formula tells us how to decompose a tensor product: if $m\ge n$ then 
$$\SB^{m}\otimes \SB^n=\SB^{m+n}\oplus \SB^{m+n-2}\oplus \cdots \oplus \SB^{m-n}.$$
The decomposition involves contraction with the skew form on $\SB$ and it follows then that $\SB^n\otimes \SB^n=\SB^{2n}\oplus \SB^{2n-2}\oplus\cdots$ 
and the skew part $\Lambda^2 \SB^n=\SB^{2n-2}\oplus \SB^{2n-6}\oplus \cdots$. 

The generators of the cohomology $H^*(G)$ have degrees $d_i=2\lambda_i+1$ where $\lambda_i$ are the exponents of the Lie algebra. For completeness we list them:
\vskip .25cm
\noindent $A_{\ell}: 1,2,3,\dots,\ell,\quad
B_{\ell}: 1,3,5,\dots, 2\ell-1,\quad C_{\ell}: 1,3,5,\dots, 2\ell-1.$

\noindent $D_{\ell}\, (\ell\,\, \mathrm {odd}): 1,3,5,\dots, 2\ell-3,\quad  F_4: 1,5,7,11,\quad G_2: 1,5.$

\noindent $E_6: 1,4,5,7,8,11,\quad E_7: 1,5,7,9,11,13,17,\quad E_8: 1,7,11,13,17,19,23,29.$
\vskip .25cm
In this list for each group  the exponents are distinct, but for $D_{\ell}$ where $\ell$ is even the exponent $\ell-1$ occurs twice. In terms of $SO(4n)$ characteristic classes the two invariants can be taken to be  the Euler class and a Pontryagin class of the same degree. The generators are not unique, just as we can take a basis of invariant polynomials for $SU(n)$ as $\tr a^k$ ($k=2,\dots,n$) or the coefficients of $\det(\lambda-a)$. 

Kostant's theorem \cite{K} tells us that under the action of the principal three-dimensional subgroup, which is unique up to conjugation,  
$\lie{g}=V_1\oplus V_2\oplus \cdots \oplus V_{\ell}$ where $V_i\cong \SB^{2\lambda_i}$. Clearly $\lambda_1=1$ gives the Lie algebra of the subgroup. 

As an example, the irreducible representation $\SB^n$ defines a homomorphism $SU(2)\rightarrow SU(n+1)$ whose image is the principal three-dimensional subgroup and the Lie algebra  $\lie{su}(n+1)$ is isomorphic to the trace zero elements in $\Hom (\SB^n,\SB^n)\cong \SB^n\otimes \SB^n$. The Clebsch-Gordon formula gives $\SB^2\oplus \cdots\oplus \SB^{2n}$ as the decomposition  $V_1\oplus V_2\oplus \cdots \oplus V_{\ell}$.

\section{Critical points}
Given an  invariant form $\varphi_i$ of degree $d_i$ we can evaluate it on an oriented $d_i$-dimensional subspace of $\lie{g}$ to obtain a function $f_i$ on the oriented Grassmannian ${ {\widetilde Gr}}(d_i,\lie{g})$ of such subspaces.
\begin{thm} The function $f_i$ has a critical point at $[V_i]$.
\end{thm}
\begin{proof} Using the metric on the Grassmannian, the gradient of $f_i$ at $[V_i]$ is a tangent vector which, by virtue of the adjoint invariance of $\varphi_i$, is invariant under the action of $SU(2)$ which stabilizes $[V_i]$. The tangent space of the Grassmannian at $[V_i]$ is isomorphic to $\Hom(V_i,\lie{g}/V_i)$, but as we have seen, except for the case $D_{\ell}$ where $\ell$ is even, the exponents are distinct and so the irreducible $V_i$ does not occur in the decomposition of $\lie{g}/V_i$. By $SU(2)$-invariance, the homomorphism is   zero and so the gradient is zero. It therefore remains to consider the case of $SO(4n)$. 

The principal three-dimensional subgroup in $SO(4n)$ acts reducibly on $\R^{4n}$. It is the representation $1\oplus \SB^{4n-2}$ and so 
$\lie{g}\cong \Lambda^2(1\oplus \SB^{4n-2})=\SB^{4n-2}\oplus \Lambda^2(\SB^{4n-2})$. Denote by $V$ the first subspace here. Using the Clebsch-Gordon decomposition we have 
$\Lambda^2(\SB^{4n-2})=\SB^{8n-6}\oplus  \SB^{8n-10}\oplus \cdots\oplus \SB^2$
which contains a copy of  $\SB^{4n-2}$  which we call $V'$.

If $e_0, e_1,\dots $ is an orthonormal basis of $1\oplus \SB^{4n-2}$ with $e_0$ spanning the trivial component then $(e_0, e_1,\dots)\mapsto (-e_0, e_1,\dots)$ is an orientation-reversing involution $\sigma$ commuting with $SO(3)$ and acting as $-1$ on $V$  and $+1$ on $V'$. The invariant polynomial on $\lie{so}(4n)$ defined by the Pfaffian $\sqrt{\det a}$ changes sign under change of orientation so it defines an invariant form $\varphi$ such that $\sigma^*\varphi=-\varphi$, hence $\varphi$ evaluated on $V'$ is zero since $\sigma=1$ there. We therefore associate $V$ to $\varphi$ and $V'$ to $\varphi'$, defined by the Pontryagin class, and consider the corresponding functions $f,f'$. Pontryagin classes are of course orientation-independent. The function $f'$ is $\sigma$-invariant and so its gradient at $[V']$ is an invariant element of $\Hom(V',V)$, but  the action here is $-1$, so the gradient vanishes and this is a critical point. The case of $f$ is similar, taking into account the fact that $\sigma$ changes orientation on $V$. 
\end{proof}

\section{Groups acting on spheres}
\subsection{The invariant forms}
We focus now on a family of covariant constant forms which arise geometrically. If a simple group $G$ acts transitively on an 
odd-dimensional sphere then we have the projection $p:G\rightarrow S^{2m+1}=G/H$ and averaging over $G$ the pull-back $p^*\omega$ of the volume form on $S^{2m+1}$ gives an invariant $(2m+1)$-form. Since $p^*\omega$ is $H$-invariant this is equivalent to averaging over the sphere as in \cite{XL}. We know in advance that this form is non-zero for, by \cite{Kudo} (see also \cite{Mat}), the stabilizer $H$ is not homologous to zero and so the cohomology class $[p^*\omega]\ne 0$. 

The groups acting transitively on spheres are well-known, especially from their appearance as special holonomy groups. For a simple group $G$ and an odd-dimensional sphere we have:
$$SO(2n),\quad SU(n),\quad Sp(n),\quad Spin(7)\subset SO(8),\quad Spin(9)\subset SO(16).$$ 
A universal multiple of the invariant form which the averaging produces can be labelled by a characteristic class which restricts to zero in the cohomology $H^*(B_H)$ of the classifying space of the stabilizer $H$ of the action. The group $H$ stabilizes  a vector in an even-dimensional space so this is the Euler class for $SO(2n)$, the Chern class $c_n$ for $SU(n)$, the Chern class $c_{2n}$ for $Sp(n)\subset SU(2n)$. The last two examples in the list are stabilizers of a vector in the spin representation  and expressing the Euler class for the spin representation in terms of the basic weights  gives multiples of  $p_1^2-4p_2$ for $Spin(7)$ and $p_1^4-8p_1^2p_2+16p_2^2-64p_4$ in the  case of $Spin(9)$ (see also \cite{F}). 

We want to prove that the invariant form is non-zero on the component $\SB^{2m}\subset \lie{g}$, the tangent space at the identity.  As in \cite{XL}, the translate of $p^*\omega$ from a general point $g$ with $p(g)=v\in S^{2m+1}\subset \R^{2m+2}$ to the identity gives a form on the Lie algebra which, evaluated on $(a_1,\dots, a_{2m+1})$, $a_i\in \lie{g}$, 
 is $\det(v, a_1v, a_2v,\dots, a_{2m+1}v)$. If $(a_1,\dots, a_{2m+1})$ forms a basis for $\SB^{2m}$ and this is nonnegative and not identically zero for all $v$ in the sphere, then the average will be positive and the invariant form will be nonzero. We proceed to consider the different cases.
 
  \subsection{The case $SO(2n)$}\label{SO}
As noted above, the principal 3-dimensional subgroup in this case arises from a reducible representation $1\oplus \SB^{2n-2}$ and the subspace $V_i\subset \lie{so}(2n)$ of dimension $2n-1$ is spanned by $a_i=e_0\otimes e_i-e_i\otimes e_0$ for $1\le i\le 2n-1$. Then 
$a_i(v)=v_ie_0-v_0e_i$ and, since $\Vert v\Vert^2=1$, 
$$v\wedge a_1v\wedge \cdots \wedge a_{2n-1}v=v_0^{2n-2}e_0\wedge e_1\wedge \cdots\wedge e_{2n-1}.$$
This is non-negative hence the average is non-zero.

This formula is Example 3.7 in \cite{XL}, where Lemma 3.5 in that paper shows that in $\lie{so}(2n)$ for general $a_i$
\begin{equation}
\det(v, a_1v, a_2v,\dots, a_{2n-1}v)=\Vert v \Vert^2 Q_{2n-2}(v)
\label{Q}
\end{equation} 
where $Q_{2n-2}(v)$ is homogeneous in $v$ of degree $2n-2$. In our situation  where $a_1,\dots, a_{2n-1}$ span one of the spaces $V_i$, this will be an invariant of the $SU(2)$ action on $\R^{2n}$ and the focus of our attention in the other cases. 
\subsection{The case $SU(n)$}\label{SU}
Here the principal three-dimensional subgroup is the action of $SU(2)$ in its irreducible representation $\SB^{n-1}$, and so its image in $SU(n)$ is a copy of $SU(2)$ for $n$ even and $SO(3)$ for $n$ odd. The $2n-1$-dimensional subspace $V_i$ is  $\SB^{2n-2}$ and so we have an inclusion 
$$\SB^{2n-2}\subset \Hom(\SB^{n-1},\SB^{n-1})\cong \SB^{n-1}\otimes \SB^{n-1}$$
and we can recognize this from the Clebsch-Gordon formula. 

In terms of polynomials $p(z)$ it is the adjoint of the multiplication map, but a more convenient description is to identify $\SB^m$ with $H^0(\PP^1,{\mathcal O}(m))$, holomorphic sections of the line bundle of degree $m$ on the projective line.  Since each $\SB^m$ has either a nondegenerate skew or symmetric form we also have an invariant identification $\SB^m\cong H^1(\PP^1,{\mathcal O}(-m-2))$ by Serre duality. Then we have a natural tensor product map 
$$H^1(\PP^1,{\mathcal O}(-2n))\otimes H^0(\PP^1,{\mathcal O}(n-1))\rightarrow H^1(\PP^1,{\mathcal O}(-n-1))\cong H^0(\PP^1,{\mathcal O}(n-1))$$
which realizes the map $\SB^{2n-2}\otimes \SB^{n-1}\rightarrow \SB^{n-1}$. This is the action of $V_i\subset \lie{su}(n)$ on $\C^n$.

Consider first the case where $n=2m+1$ is odd, then $\SB^{n-1}=\SB^{2m}$ is even and has a real structure and so we can write a complex vector $v=v_1+iv_2$ where $v_1,v_2$ are real. Of course $SU(n)$ does not preserve the real structure, only the three-dimensional subgroup does. Now $\SB^{4m}\subset \SB^{2m}\otimes \SB^{2m}$ is symmetric and real  and elements of $V_i\subset \lie{su}(2m+1)$ are of the form $iA$ for a real symmetric matrix $A$. 

As in equation \ref{Q}, we are concerned with the expression $v\wedge a_1v\wedge \cdots \wedge a_{2n-1}v$ considering $\C^n$ as a real vector space where the $a_j$ lie in $V_i$. This vanishes when some linear combination of the $a_i$ has $v$ as a real eigenvector. But the $a_i$ are skew adjoint so it can only be the zero eigenvalue.
Now  each $a\in V_i$ is of the form $iA$ for $A$ real, and so $iA(v_1+iv_2)=-Av_2+iAv_1$ and if this vanishes then $Av_1=0=Av_2$.    

Represent $A$ as an element $[\alpha]$ of $H^1(\PP^1,{\mathcal O}(-2n))$ and $v_1$ as a section $s$ of ${\mathcal O}(n-1)$ then $Av_1=0$ has an interpretation in algebraic geometry: consider the exact sequence of sheaves 
$$0\rightarrow {\mathcal O}(-2n)\stackrel{s}\rightarrow  {\mathcal O}(-n-1)\rightarrow {\mathcal O}_{D}(-n-1)\rightarrow 0$$
where $D$ is the divisor of zeros of $s$. Then the long exact cohomology sequence gives
$$0\rightarrow H^0(D, {\mathcal O}_{D}(-n-1))\stackrel{\delta}\rightarrow H^1(\PP^1, {\mathcal O}(-2n))\stackrel{s}\rightarrow H^1(\PP^1,{\mathcal O}(-n-1))\rightarrow 0$$
so that $[\alpha] s=0$ if and only if $[\alpha]=\delta t$ for a section $t$ of ${\mathcal O}(-n-1)$ on the zero-dimensional  cycle $D$. 

Let $s_1$ and $s_2$ be two sections representing $v_1,v_2$ which have a common zero $x$ then the cycles $D_1,D_2$ intersect and taking $t$ as a section of ${\mathcal O}(-n-1)$ on $x$ defines $\delta(t)=[\alpha]$ which annihilates both $s_1$ and $s_2$. Hence $[\alpha]$ represents a linear combination of $a_j$ such that $v\wedge a_1v\wedge \cdots \wedge a_{2n-1}v$ vanishes when $v=v_1+iv_2$ and $v_1,v_2$ are represented by $s_1,s_2$ which have a common zero. These are  polynomials $p_1(z),p_2(z)$ of degree $n-1$ and 
the condition for a common zero is the vanishing of the resultant 
$$R(p_1,p_2)=a_0^{n-1}b_0^{n-1}\prod_{i,j}(\lambda_i-\mu_j)=a_0^{n-1}\prod_i p_2(\lambda_i)$$
where $\lambda_i, \mu_j$ are the roots of $p_1(z)=a_0z^{n-1}+\cdots+a_{n-1}, p_2(z)=b_0z^{n-1}+\cdots+b_{n-1}$. This is a polynomial in $v=v_1+iv_2$ homogeneous of degree $2n-2$. Its vanishing implies $Q_{2n-2}$  from equation (\ref{Q}) vanishes, but these two invariant polynomials have the same degree and the  resultant is irreducible hence they are multiples of each other.

The real structure on $\SB^{n-1}$ is inherited from the quaternionic structure of $\SB$ so a real polynomial of degree $2m$ satisfies $p(-1/\bar z)=\bar z^{-2m}\overline{p(z)}$
and there is a free involution $\lambda\mapsto -1/\bar\lambda$ on the roots of $p$.  Let $\lambda_1,\dots, \lambda_m, -1/\bar \lambda_1,\dots, -1/\bar\lambda_m$ be the roots of $p_1$, then
$$ R(p_1,p_2)=a_0^{2m}\prod_{i=1}^m p_2(\lambda_i)p_2(-1/\bar\lambda_i)=(a_0\prod_{i=1}^m\bar \lambda_i^{-1})^{2m}\prod _{i=1}^m\vert p_2(\lambda_i)\vert^2.$$
Reality implies  $a_{2m}=\bar a_0$ so that the product of the roots is $\bar a_0/a_0$ and $a_0\prod_1^m\bar \lambda_i^{-1}$ is real. Hence the resultant is non-negative and averaging gives a non-zero evaluation of the form. 

When $n=2m$ is even, $\SB^{2m-1}$ has a complex symplectic structure and a quaternionic structure: an antilinear involution $J$ with $J^2=-1$. Then $\SB^{4m-2}\subset  \SB^{2m-1}\otimes \SB^{2m-1}$ is symmetric which places it in the Lie algebra of complex symplectic transformations. But it is also real and so commutes with $J$. In this case if a linear combination of the $a_i$ annihilates $v$ it annihilates $Jv$ so we again have a 2-dimensional kernel and the criterion is the vanishing of the resultant of two polynomials -- $p$ and its transform $p^*$ by $J$ where $p^*(z)=z^{2m-1}\overline{p(-1/\bar z)}$. Then the resultant  $R(p,p^*)$ is
$$(a_0\bar a_{2m-1})^{2m-1}\prod_{i,j}(\lambda_i+\bar \lambda_j^{-1})=(a_0\bar a_{2m-1})^{2m-1}\prod_i(\vert \lambda_i\vert^2+1)\prod_{i< j}\vert \lambda_i\bar \lambda_j+1\vert^2 \left(\prod _j \bar \lambda_j^{-1}\right)^{2m-1}$$
and since $\prod _j \bar \lambda_j=-\bar a_{2m-1}/\bar a_0$ this expression is non-positive. Again the average is non-zero.
\subsection{The case $Sp(n)$}
The group $Sp(n)\subset SU(2n)$ is the subgroup which commutes with a quaternionic structure $J$ and we have just observed that the appropriate $V_i$ does just that, so that it lies in the Lie algebra $\lie{sp}(n)$. The result follows from the previous section. 
 
\subsection{The case $Spin(7)$}
Here the principal three-dimensional subgroup of $Spin(7)$ projects to the principal one in $SO(7)$. This is the irreducible representation   $\SB^6$ and from the characters we deduce that the 8-dimensional spin representation is $1\oplus \SB^6$. This means that the subgroup  fixes a spinor and so lies in the stabilizer $G_2$.

 The Lie algebra of $G_2$ decomposes as $\SB^2\oplus \SB^{10}$ and $\lie{so}(7)= \SB^2\oplus \SB^6\oplus \SB^{10}$ with respect to the same 3-dimensional group. It follows that  $\SB^6$ is the orthogonal complement of $\lie{g}_2$. Translated around $Spin(7)$ this is the horizontal subspace for the fibration $p:Spin(7)\rightarrow S^7$. This is a Riemannian submersion  so  $p^*\omega$ is always non-zero on this subspace.

\subsection{The case $Spin(9)$}
The defining $9$-dimensional representation is here $\SB^8$ and, from the characters again, the 16-dimensional spin representation is $\SB^{10}\oplus \SB^4$. In the Lie algebra $\lie{so}(9)\cong \Lambda^2 \SB^8$ the $15$-dimensional component is $\SB^{14}$ and we are concerned with its action on  $\SB^{10}\oplus \SB^4$. Since $\Lambda^2(\SB^{10}\oplus \SB^4)\cong \Lambda^2(\SB^{10})\oplus (\SB^{10}\otimes \SB^4)\oplus \Lambda^2\SB^4$ there 
are  copies of $\SB^{14}$ in the first two summands and the action is a linear combination of the two. 

We consider again when a linear combination of $a_1,\dots, a_{15}\in V_i$ has a non-trivial kernel. Suppose $(p,q)\in \SB^{10}\oplus \SB^4$ are polynomials in the kernel of $a\in \SB^{14}$ then we may write this as $(Ap+Bq, -B^Tp)=0$ where   $a=(A,B)\in  \Lambda^2(\SB^{10})\oplus (\SB^{10}\otimes \SB^4)$. Now $B^T:\SB^4\rightarrow \SB^{10}$ is given by the map
$$H^1(\PP^1,{\mathcal O}(-16))\otimes H^0(\PP^1,{\mathcal O}(4))\rightarrow H^1(\PP^1,{\mathcal O}(-12))$$
as in Section \ref{SU} and $B$ by the map 
$$H^1(\PP^1,{\mathcal O}(-16))\otimes H^0(\PP^1,{\mathcal O}(10))\rightarrow H^1(\PP^1,{\mathcal O}(-6))$$
for a class $[\beta]\in H^1(\PP^1,{\mathcal O}(-16))\cong \SB^{14}$. If  $p,q$ have a common zero then there exists $[\beta]$ with $Bp=0, B^Tq=0$ represented by a class supported at a single point in $\PP^1$, the common zero. If we take this point to be $z=0$ then $[\beta]$ can be identified with  the polynomial $z^{14}\in \SB^{14}$. 

Consider now $A:\SB^{10}\rightarrow \SB^{10}$ defined by $z^{14}$. This consists of contracting in $\SB^{14}\otimes \SB^{10}$ seven pairs of terms and symmetrizing. If $p$ vanishes at $0$, contraction with $z^{14}$ vanishes also. We deduce that the vanishing of the resultant $R(p,q)$ is a condition for the existence of $a\in V_i$ which annihilates $(p,q)$. This is a polynomial in the coefficients of degree $4+10=14$. But $Q_{2n-2}(v)=Q_{14}(v)$ in (\ref{Q}) is of degree 14 and so $Q_{14}$ is a multiple of the resultant of two real polynomials $u,v$ of even degrees $4,10$. As in Section \ref{SU}, this is non-negative. 
\section{Conclusion}
We have shown that in certain degrees and certain groups there exists an invariant form which is nonvanishing on $V_i$. This is true for $V_1$ for any $G$, where of course the Cartan three-form restricts non-trivially to any three-dimensional subgroup, not just the principal one.  When $G$ has rank  $\ell=2$ we have $\lie{g}=V_1\oplus V_2$,  an orthogonal decomposition, and the Hodge star of the Cartan 3-form calibrates $V_2$ so all cases are covered. Another example is the group $SU(4)$ which acts transitively on $S^7$ and also on $S^5$ under the homomorphism $SU(4)\rightarrow SO(6)$, identifying $SU(4)$ with $Spin(6)$, so we have forms in all degrees $3,5,7$ in this case, but for higher rank the arguments in this article only relate to a restrictive number of   forms. 

\section{Polyvector fields}
We conclude with a brief discussion of the origin in \cite{H} of the conjecture that for each subspace $V_i$ there is an invariant form $\varphi_i$ on $\lie{g}$ which restricts nontrivially. The context is a Riemann surface $C$ of genus $g>1$ and the moduli space $M$ of stable holomorphic principal $G^c$-bundles $P$ on $C$ for a  complex simple Lie group $G^c$. The cotangent space at a point of $M$ is isomorphic to $H^0(C, \ad(P)\otimes K)$ where $K$ is the canonical bundle and evaluating an invariant polynomial $p$ of degree $k$ defines a holomorphic section of $K^k$ on $C$. Taking the dual of $H^0(C,K^k)$ this yields a map $H^1(C, K^{1-k})\rightarrow H^0(M, S^kT)$ which is well-known to be injective and to generate holomorphic sections of the symmetric powers $S^kT$ of the tangent bundle which commute using the Schouten-Nijenhuis bracket \cite{H2}, or equivalently define Poisson-commuting functions on the cotangent bundle $T^*M$.

If we now use an invariant \emph {alternating} form $\varphi$ of degree $d$ then evaluation  yields a section of $K^{d}$ and dually we have a map $H^1(C, K^{1-d})\rightarrow H^0(M, \Lambda^dT)$ into the space of polyvector fields on $M$ and these also Schouten-commute \cite{H}. However, whereas using the spectral curve one can see that in the symmetric case the map is injective, for the skew-symmetric case this is not apparent. Instead consider the $G^c$-bundle associated to a rank $2$ stable bundle $V$  by the principal homomorphism $SL(2,\C)\rightarrow G^c$ then we can restrict  a form $\varphi_i$ to the subspace $H^0(C,S^{2\lambda_i}V\otimes K)\subset H^0(C, \ad(P)\otimes K)$. By Riemann-Roch this has dimension $(2\lambda_i+1)(g-1)$ so 
if the conjecture held then  choosing $n=2\lambda_i+1$ holomorphic sections $s_j$ with $s_1\wedge s_2\wedge \cdots \wedge s_n$ not identically zero, we could deduce that $\varphi_i$ gives a nonzero section of $K^{d_i}$. There may of course be simpler ways of achieving this.

 \vskip 2cm
  Mathematical Institute
  
  Woodstock Road
  
  Oxford OX2 6GG
  
  UK

\end{document}